\documentclass[reqno]{amsart}

\title[Entropy, regularity and rigidity]{On entropy, regularity and rigidity for convex representations of hyperbolic manifolds}
\author{A. Sambarino}
\thanks{The author was partially supported by
the European Research Council under the {\em European Community}'s seventh Framework Program (FP7/2007-2013)/ERC {\em grant agreement} n° FP7-246918.}
\usepackage{marvosym}
\usepackage{psfrag}
\usepackage[english]{babel}
\usepackage{mathrsfs}
\usepackage{verbatim}
\usepackage[ansinew]{inputenc}
\usepackage{amsmath,amsthm,amssymb,amscd}
\usepackage[all,cmtip]{xy}
\usepackage{xspace}
\usepackage[dvips]{color}
\usepackage{epsfig}
\usepackage{pb-diagram}
\usepackage{amsfonts}
\usepackage{graphicx}
\usepackage[mathcal]{eucal}

\usepackage{perpage}
\MakePerPage{footnote}

\newcommand{\R}{\mathbb{R}}

\renewcommand{\H}{\mathbb{H}}

\renewcommand{\P}{\mathbb{P}}

\renewcommand{\/}{\backslash}
\renewcommand{\k}{\kappa}

\newcommand{\eps}{\varepsilon}

\newcommand{\upla}[3]{#1_{#2},\ldots,#1_{#3}}

\newcommand{\G}{\Gamma}
\renewcommand{\l}{\ell}
\newcommand{\<}{\left<}
\renewcommand{\>}{\right>}
\newcommand{\E}{\Sigma}
\newcommand{\scr}{\mathscr}
\newcommand{\g}{\gamma}
\renewcommand{\a}{\t}

\newcommand{\z}{\zeta}
\newcommand{\w}{\widetilde}
\newcommand{\bord}{\partial_\infty}
\newcommand{\vo}[1]{\overline{#1}}

\newcommand{\cone}{\scr L}

\newcommand{\bus}{\sigma}

\newcommand{\M}{\sf M}
\renewcommand{\d}{\delta}
\newcommand{\grupo}{\Delta}
\newcommand{\om}{\omega}
\newcommand{\posgen}{\scr F^{(2)}}
\renewcommand{\t}{\theta}

\renewcommand{\sf}{\mathsf}
\newcommand{\UG}{{\sf{U}}\G}
\newcommand{\hX}{h_\G}
\newcommand{\h}{{\sf{H}}}
\newcommand{\A}{{\sf{A}}}
\newcommand{\LL}{{\sf{L}}}
\newcommand{\II}{{\bf I}}

\renewcommand{\L}{\Lambda}

\newcommand{\cal}{\mathcal}
\renewcommand{\frak}{\mathfrak}

\DeclareMathOperator{\ii}{i}
\DeclareMathOperator{\Lim}{L}

\DeclareMathOperator{\SL}{SL}
\DeclareMathOperator{\PSL}{PSL}

\DeclareMathOperator{\Leb}{Leb}

\DeclareMathOperator{\clase}{C}

\DeclareMathOperator{\Ad}{Ad}

\DeclareMathOperator{\isom}{Isom}
\DeclareMathOperator{\SO}{SO}
\DeclareMathOperator{\Sp}{Sp}
\DeclareMathOperator{\Ge}{G_2}
\DeclareMathOperator{\PGL}{PGL}

\DeclareMathOperator{\PSO}{PSO}

\DeclareMathOperator{\tope}{top}

\DeclareMathOperator{\CAT}{CAT}

\newtheorem*{teo1}{Theorem A}
\newtheorem*{teoB}{Theorem B}
\newtheorem*{teoC}{Theorem C}
\newtheorem*{teoD}{Theorem D}

\newtheorem{teo}{Theorem}[section]
\newtheorem{cor}[teo]{Corollary}

\newtheorem*{cor2}{Corollary}
\newtheorem{lema}[teo]{Lemma}

\newtheorem{prop}[teo]{Proposition}

\theoremstyle{definition}

\newtheorem{defi}[teo]{Definition}

\newtheorem{obs}[teo]{Remark}

\theoremstyle{remark}

\begin{document}

\begin{abstract} Given a convex representation $\rho:\G\to\PGL(d,\R)$ of a convex cocompact group $\G$ of $\isom_+\H^k,$ we find upper bounds for the quantity $\alpha h_\rho,$ where $h_\rho$ is the entropy of $\rho$ and $\alpha$ is the H\"older exponent of the equivariant map $\bord\G\to\P(\R^d).$ We also give rigidity statements when the upper bound is attained. We then prove that if $\rho:\pi_1\E\to\PSL(d,\R)$ is in the Hitchin component then $\alpha h_\rho\leq 2/(d-1)$ (where $\alpha$ is the H\"older exponent of Labourie's equivariant flag curve) with equality if and only if $\rho$ is Fuchsian.\end{abstract}

\maketitle

\tableofcontents

% !TEX root = regularityandrigidity.tex

\section{Introduction}

Consider a $\CAT(-1)$ space $X.$ Its visual boundary $\bord X$ is equipped with a natural metric called a visual metric. This metric depends on the choice of a point in $X$ and different points induce bi-Lipschitz equivalent metrics.

Consider now a convex cocompact action of a hyperbolic group $\G$ on $X.$ An important invariant for this action is the Hausdorff dimension $\hX,$ for a (any) visual metric, of the limit set $\Lim_\G$ of $\G$ on the visual boundary $\bord X$ of $X.$

Several rigidity statements have been found concerning lower bounds on this Hausdorff dimension. For example, Bourdon \cite{bourdon2} proved that if $\G=\pi_1\M,$ where $\M$ is a closed $k$-dimensional manifold modeled on $\H^k,$ then $\hX\geq k-1$ with equality only if there is a totally geodesic copy of $\H^k$ on $X$ preserved by $\G.$ We refer the reader to Courtois \cite{courtoislibro} for a more detailed exposition on this problem.

Given two convex cocompact actions $\rho_i:\G\to\isom X_i$ $i=1,2$ on $\CAT(-1)$ spaces $X_i,$ there is an obvious relation between the Hausdorff dimensions of their limit sets. Let $\xi:\Lim_{\rho_1\G}\to \Lim_{\rho_2\G}$ be the H\"older-continuous equivariant map. From the definition of Hausdorff dimension one obtains \begin{equation}\label{eq:alpha}\alpha h_{\rho_2}\leq h_{\rho_1},\end{equation} when $\xi$ is \emph{$\alpha$-H\"older}, i.e. when $d(\xi(x),\xi(y))\leq K d(x,y)^\alpha$ for some $K>0,$ and all $x,y.$ Denote by $$\alpha_{(\rho_1,\rho_2)}=\sup\{\alpha\in\R_+^*: \xi \textrm{ is $\alpha$-H\"older}\}.$$ Remark that $\xi$ is not necessarily $\alpha_{(\rho_1,\rho_2)}$-H\"older. Incidentally, we prove the following proposition. For a non-torsion $\g\in\G,$ denote by $$|\g|=\inf_{p\in X}d_X(p,\g p),$$ the length of the closed geodesic of $\G\/X$ determined by the conjugacy class $[\g]$ of $\g.$ 

\begin{prop}\label{prop:1.1} Consider two convex cocompact actions of $\G$ on $\CAT(-1)$-spaces, $\rho_i:\G\to\isom(X_i),$ $i\in\{1,2\},$ such that $\alpha_{(\rho_1,\rho_2)} h_{\rho_2}=h_{\rho_1}.$ Then for every non-torsion $\g\in\G,$ one has $$|\rho_2\g|=\alpha_{(\rho_1,\rho_2)}|\rho_1\g|.$$
\end{prop}

Deciding if an equation such as $|\rho_2\g|=c|\rho_1\g|$ for some $c>0$ and all non-torsion $\g\in\G$ determines the action $\rho_1,$ is a difficult problem known as \emph{the marked length spectrum problem} (when $c=1)$. Besides certain situations such as negatively curved closed surfaces (treated by Otal \cite{otal}) or if $\rho_1$ is cocompact in $\H^n$ (treated by Bourdon \cite{bourdon2} and Hamenst\"adt \cite{ursula}) little is known. 

The following is a corollary of Theorem B below.

\begin{cor} Let $\rho_1:\G\to\isom_+\H^k$ be a Zariski dense convex cocompact action, and consider a convex cocompact action $\rho_2:\G\to\isom_+\H^n$ such that $\alpha_{(\rho_1,\rho_2)} h_{\rho_2}=h_{\rho_1},$ then $\xi$ is the induced map on the boundary, of an equivariant isometric embedding $\H^k\to\H^n.$
\end{cor}

This is to say, $\rho_2\G$ preserves a totally geodesic copy of $\H^k$ in $\H^n$ and moreover, the action of $\rho_2\G$ on this geodesic copy, is conjugated by an isometry to $\rho_1.$ This provides a rigidity statement for Schottky groups, for example.

%One can define an asymmetric distance on the space $$\{\rho:\G\to\isom\H^k:\textrm{ convex cocompact}\}/\isom\H^k,$$ by $$d(\rho_1,\rho_2)=\log\frac{h_{\rho_1}}{\alpha_{(\rho_1,\rho_2)}h_{\rho_2}}.$$ This is closely related to Thurston asymmetric distance for closed hyperbolic surfaces, as it is implicitly suggested by Sorvali \cite{sorvali}.

The main purpose of this work is to extend inequality (\ref{eq:alpha}) for convex representations, and give rigidity results when the equality holds. In order to do so, we will exploit the well known fact that $\hX$ is also a dynamical invariant.

Consider the geodesic flow of $\G\/X,$ $\phi=(\phi_t:\G\/{\sf U}X\to\G\/{\sf U} X)_{t\in\R}.$ The fact that $\G$ is convex cocompact, is equivalent to the fact that the non-wandering set of $\phi,$ denoted from now on $\UG,$ is compact. Moreover, $\phi|\UG$ has very nice dynamical properties coming from the negative curvature of $X,$ namely it is a metric Anosov flow (see Definition \ref{defi:anosovtop}).

The topological entropy of $\phi$ coincides with the Hausdorff dimension $\hX$ (Sullivan \cite{sullivan}, see also Bourdon \cite{bourdon}), and can be computed by counting how many periodic orbits $\phi$ has:  \begin{equation}\label{equation:entropia} \hX=\lim_{t\to\infty} \frac{1}t\log\#\{[\g]\in[\G]\textrm{ non-torsion}:|\g|\leq t\}.\end{equation}

\begin{defi} We will say that a representation $\rho:\G\to\PGL(d,\R)$ is \emph{convex} if there exist a $\rho$-equivariant H\"older-continuous map $$(\xi,\xi^*):\Lim_\G\to\P(\R^d)\times\P((\R^d)^*)$$ such that if $x,y\in\bord\G$ are distinct, then $\xi(x)\oplus\ker\xi^*(y)=\R^d.$
\end{defi}

Different notions of entropy can be defined for a convex representation by analogy with equation (\ref{equation:entropia}). For $g\in\PGL(d,\R)$ denote by $\lambda_1(g)$ the logarithm of the spectral radius of $g.$ The \emph{spectral entropy} of a convex representation $\rho:\G\to\PGL(d,\R)$ is defined by $$h_\rho=\lim_{t\to\infty} \frac1t \log\#\left\{[\g]\in[\G]\textrm{ non-torsion}:\lambda_1 (\rho\g)\leq t\right\},$$ and the \emph{Hilbert entropy} of $\rho$ is defined by $$\h_\rho=\lim_{t\to\infty} \frac1t \log\#\left\{[\g]\in[\G]\textrm{ non-torsion}:{\displaystyle\frac{\lambda_1 (\rho\g)-\lambda_d(\rho\g)}2}\leq t\right\},$$ where $\lambda_d(\rho\g)$ is the $\log$ of the modulus of the smallest eigenvalue of $\rho\g.$ One has the following proposition.

\begin{prop}[\cite{quantitative}, see also \cite{presion}]\label{prop:finito} The spectral entropy of an irreducible convex representation of a (finitely generated non-elementary) hyperbolic group, is finite and positive.
\end{prop}

If $V$ is a finite dimensional vector space we will consider the distance $d_\P$ on $\P(V),$ induced by a Euclidean metric on $V.$ An important remark is that the entropy of a convex representation is not necessarily the Hausdorff dimension of $\xi(\bord\G)$ (see Remark \ref{obs:hff} below). Our first result is the following:

\begin{teo1} Let $\G$ be a convex cocompact group of a $\CAT(-1)$ space $X$ and let $\rho:\G\to\PGL(d,\R)$ be an irreducible convex representation with $d\geq3.$ Then $$\alpha h_\rho\leq \hX\textrm{ and } \alpha\h_\rho\leq\hX,$$ when $\xi$ is $\alpha$-H\"older. 
\end{teo1}

Observe that the dimension $d$ of $\R^d$ does not appear in the inequality. 

Consider $\Ad:\PGL(d,\R)\to\PGL(\frak{sl}(d,\R))$ the Adjoint representation. If $\rho:\G\to\PGL(d,\R)$ is an irreducible convex representation then $\Ad\rho:\G\to\PGL(\frak{sl}(d,\R))$ is not necessarily irreducible but there is a natural subspace $V_\rho\subset\frak{sl}(d,\R)$ such that $$\A_\rho=\Ad\rho|V_\rho:\G\to\PGL(V_\rho)$$ is again irreducible and convex (see Lemma \ref{lema:adjunta}). The representation $\A_\rho$ will be refered to as \emph{the irreducible adjoint representation of $\rho$,} and will play an important role on understanding rigidity for Hilbert's entropy.

A simple computation shows that the Hilbert entropy of $\rho$ is related to the spectral entropy of $\A_\rho,$ namely $\h_\rho=2h_{\A_\rho}.$ Nevertheless, applying this relation to the first inequality in Theorem A, gives the bad upper bound $\alpha\h_\rho\leq 2\hX.$

\section{Examples}

There are three examples of irreducible convex representations of $\G$ of particular interest.

Recall that the group $\PSO(1,k),$ of projective transformations preserving a bilinear form of signature $(1,k),$ is isomorphic to the orientation preserving isometry group $\isom_+\H^k,$ of the $k$-dimensional hyperbolic space. Throughout this work, we will refer to the representation $\vo\phi:\isom_+\H^k \to \PSO(1,k)$ (or any of its conjugates $g\vo\phi g^{-1}$ with $g\in\PGL(k+1,\R)$) as \emph{the Klein model} of $\H^k.$ 

\begin{obs}\label{obs:lipschitz} The Klein model of $\H^k$ induces an equivariant map $\bord\H^k\to\P(\R^{k+1}).$ This equivariant map is a bi-Lipschitz homeomorphism onto its image.
\end{obs}

\subsection{Benoist Representations} If $\rho:\G\to\PGL(k+1,\R)$ preserves a proper open convex set $\Omega_\rho$ of $\P(\R^{k+1})$ and $\rho\G\/\Omega_\rho$ is compact, then $\rho$ is called a \emph{Benoist representation}\footnote[2]{These are also called \emph{divisible convex sets with strictly convex boundary} or \emph{strictly convex projective structures on closed manifolds}.}. Results from Benoist \cite{convexes1}, imply that Benoist representations are irreducible convex representations (see \cite{quantitative} for details). 

The Hilbert entropy of $\rho$ is the topological entropy of the geodesic flow of $\rho\G\/\Omega_\rho,$ associated to the Hilbert metric. Crampon \cite{crampon} proved that the Hilbert entropy verifies $\h_\rho\leq k-1=\dim\partial\Omega_\rho,$ and equality holds only when $\Omega_\rho$ is an ellipsoid, i.e. $\G$ acts cocompactly on $\H^k,$ and $\rho$ extends to the Klein model of $\H^k.$

Notice that $\partial\Omega_\rho=\xi(\bord\G)$ is topologically a $k-1$ dimensional sphere, hence when $\Omega_\rho$ is not an ellipsoid,$\h_\rho$ is not the Hausdorff dimension of $\xi(\bord\G).$

\subsection{Convex cocompact groups in $\H^k$}  Consider a convex cocompact group $\phi:\G\to\isom_+\H^k.$ The composition of $\phi$ with the Klein model of $\H^k$ gives rise to a convex representation $\phi':\G\to\PGL(k+1,\R).$ 

In this setting, $\phi\G$ is Zariski-dense in $\isom_+\H^k$ if and only if, up to finite index, $\phi\G$ does not have an invariant totally geodesic copy of $\H^{k-1}.$ If this is the case, the convex representation $\phi'\G$ is irreducible.

An easy computation shows that the spectral entropy of $\phi',$ and the Hilbert entropy, coincide with the topological entropy of the geodesic flow of $\phi\G\/\H^k,$ which in turn coincides with the Hausdorff dimension of the limit set $\Lim_{\phi\G}$ on $\bord\H^k,$ (Sullivan \cite{sullivan}).

Assume now that $\G=\pi_1\M$ is the fundamental group of a closed $k$-dimensional hyperbolic manifold, it is well known that $\hX=k-1.$ Consider now a convex co-compact action $\phi:\pi_1\M\to\isom_+\H^n$ with $n\geq k.$ As we explained before, Bourdon states that $h_{\phi}\geq k-1.$

In light of the last examples, one sees that a deformation of $$\pi_1\M\to\isom_+\H^k\to\PGL(k+1,\R)$$ decreases Hilbert's entropy, but on the contrary, a deformation of $$\pi_1\M\to\isom_+\H^k\to\isom\H^n$$ increases Hilbert's entropy. As a conclusion, the Hilbert entropy of a convex representation of $\pi_1\M$ may be greater or smaller than $\dim\M-1,$ nevertheless the quantity $\alpha \h$ has to remain bounded by this number. Theorem A is thus optimal in this generality.

\subsection{Hitchin representations and small deformations of exterior products} Consider a closed oriented hyperbolic surface $\E,$ and say that a representation $\rho:\pi_1\E\to\PSL(d,\R)$ is \emph{Fuchsian} if it factors as $$\rho=\tau_d\circ{\sf{f}},$$ where $\tau_d:\PSL(2,\R)\to\PSL(d,\R)$ is the irreducible linear action (unique modulo conjugation) of $\PSL(2,\R)$ on $\R^d,$ and $\sf f: \pi_1\E\to\PSL(2,\R)$ is a choice of a hyperbolic metric on $\E.$  A \emph{Hitchin component} of $\PSL(d,\R),$ is a connected component of $\hom(\pi_1\E,\PSL(d,\R))$ containing a Fuchsian representation. As Hitchin \cite{hitchin} proves, representations in the Hitchin component are irreducible.

Recall that a (complete) \emph{flag} of $\R^d$ is a collection of subspaces $\{V_i\}_{i=0}^d,$ such that $V_i\subset V_{i+1}$ and $\dim V_i=i.$ The space of flags is denoted by $\scr F.$ Two flags $\{V_i\}$ and $\{W_i\}$ are in \emph{general position}, if for every $i$ one has $$V_i\oplus W_{d-i}=\R^d.$$

Labourie \cite{labourie} proves that if $\rho:\pi_1\E\to\PSL(d,\R)$ is a representation in a Hitchin component, then there exists a $\rho$-equivariant H\"older-continuous map $\z:\bord\pi_1\E\to\scr F,$ such that the flags $\z(x)$ and $\z(y)$ are in general position when $x,y\in\bord\pi_1\E$ are distinct.

Considering thus $\xi=\z_1$ the first coordinate of $\z,$ and $\xi^*=\z_d$ the last coordinate of $\z,$ one obtains an irreducible convex representation. Moreover, let $\L^n\R^d$ be the $n$-th exterior power of $\R^d.$ An $n-$dimensional subspace is sent to a line on $\L^n\R^d,$ hence Labourie's theorem implies that the composition $\L^n\rho:\pi_1\E\to\PSL(\L^n\R^d)$ is again convex. 

Finally, if $\rho$ is Zariski-dense on $\PGL(d,\R),$ then $\L^n\rho$ is irreducible. Guichard and Wienhard \cite{olivieranna} have shown that convex irreducible representations form an open set on the space of representations. Hence small deformations of $\L^n\rho$ are still irreducible and convex.

\begin{obs}\label{obs:hff} Labourie's statement implies that if $\rho:\pi_1\E\to\PGL(d,\R)$ is a Hitchin representation, then the image $\xi(\bord\pi_1\E)$ is a curve of class $\clase^1$ (even thought the map $\xi$ is only H\"older). Hence, neither entropy of $\rho$ can be interpreted as the Hausdorff dimension of $\xi(\bord\pi_1\E).$ For example, if $\rho$ is Fuchsian, then an easy computation shows that $h_\rho=\h_\rho=2/(d-1),$ even thought the limit curve is a polynomial.
\end{obs}

\section{Rigidity statements}

For a convex representation $\rho:\G\to\PGL(d,\R)$ and a fixed action of $\G$ on a $\CAT(-1)$ space $X,$ denote by $$\alpha_\rho=\sup\{\alpha\in\R_+^*: \xi:\Lim_\G\to\P(\R^d) \textrm{ is $\alpha$-H\"older}\},$$ the ``best'' H\"older exponent of the equivariant map $\xi.$ Remark that $\xi$ is not necessarily $\alpha_\rho$-H\"older.

\begin{teoB}[Spectral entropy rigidity] Let $\G$ be a Zariski-dense convex cocompact group of $\isom_+\H^k,$ and consider a convex irreducible representation $\rho:\G\to\PGL(d,\R)$ with $d\geq3$ such that $$\alpha_\rho h_\rho=\hX.$$ Then $d=k+1,$ $\alpha_\rho=1$ and $\rho$ extends to $\vo\rho:\isom_+\H^k\to\PGL(k+1,\R)$ as the Klein model of $\H^k.$
\end{teoB}

A slight modification of the proof of Theorem B gives the following weaker statement for Hilbert's entropy.

\begin{cor}[Hilbert entropy rigidity]\label{cor:hilbert} Let $\G$ be a Zariski-dense convex cocompact group of $\isom_+\H^k,$ and consider a convex irreducible representation $\rho:\G\to\PGL(d,\R)$ with $d\geq3,$ such that $$\alpha_\rho \h_\rho=\hX.$$ Then $V_\rho=\frak{so}(1,k)$ and the adjoint irreducible representation $\A_\rho:\G\to\PGL(\frak{so}(1,k))$ extends to $\vo{\A_\rho}:\isom_+\H^k\to\PGL(\frak{so}(1,k))$ as the adjoint representation of the Klein model of $\H^k.$ 
\end{cor}

The proofs of Theorem B and Corollaries \ref{cor:hilbert} are very similar and postponed to Section \ref{section:Aigualdad}.

\subsection{Statements for hyperconvex representations}\label{subsection:hyper} 

The fact that equality in Theorem B can only hold for a representation $\rho:\pi_1\E\to\PSL(3,\R),$ suggests that the upper bound for $\alpha_\rho h_\rho$ is not optimal, for Hitchin representations on $\PSL(d,\R),$ say. We will now focus on improving the bound when more information on the representation $\rho$ is given.

Let $G$ be a real-algebraic semisimple Lie group without compact factors, $P$ a minimal parabolic subgroup of $G,$ and denote by $\scr F=G/P$ the \emph{Furstenberg boundary} of the symmetric space of $G.$

Let $K$ be a maximal compact subgroup of $G,$ let $\tau$ be the Cartan involution on $\frak g$ whose fixed point set is the Lie algebra of $K.$ Consider $\frak p=\{v\in\frak g: \tau v=-v\}$ and $\frak a$ a maximal abelian subspace contained in $\frak p.$ Let $\E$ be the set of (restricted) roots of $\frak a$ on $\frak g.$ Fix $\frak a^+$ a closed Weyl chamber and let $\E^+$ be a system of positive roots on $\E$ associated to $\frak a^+.$ Denote by $\Pi$ the set of simple roots associated to the choice $\E^+.$ 

The space $\scr F$ can be embedded in a product of projective spaces $\prod_{\a\in\Pi}\P(V_\a)$ (see Section \ref{section:reps}), we will consider the metric on $\scr F$ induced by this embedding.

The product $\scr F\times\scr F$ has a unique open $G$-orbit denoted by $\posgen.$ For example, if $G=\PGL(d,\R)$ then $\scr F$ is the space of complete flags of $\R^d,$ and $\posgen$ is the space of flags in general position.

\begin{defi} We say that a representation $\rho:\G\to G$ is \emph{hyperconvex} if there exists a H\"older-continuous equivariant map $\z:\bord\G\to\scr F,$ such that if $x,y\in\bord\G$ are distinct, then $(\z(x),\z(y))$ belongs to $\posgen.$
\end{defi}

Hyperconvex representations on $\PGL(d,\R)$ are, of course, convex. As explained before, Labourie \cite{labourie} proved that representations in a Hitchin component are hyperconvex. 

We will say that $g\in G$ is $\R$\emph{-regular} if it is diagonalizable over $\R,$ \emph{elliptic} if it is contained in a compact subgroup of $G,$ or \emph{unipotent} if all its eigenvalues are equal to 1.

Recall that Jordan's decomposition states that every $g\in G$ can be written as a product $g=g_eg_hg_u,$ where $g_e,g_h,g_u\in G$ commute, $g_e$ is elliptic, $g_h$ is $\R$-regular and $g_u$ is unipotent.

For $g\in G$ denote by $\lambda(g)\in\frak a^+$ its \emph{Jordan projection}, this is the unique element on $\frak a^+$ such that $\exp\lambda(g)$ is conjugated to the $\R$-regular element on the Jordan decomposition of $g.$

If $\rho:\G\to G$ is a hyperconvex representation, and $\varphi\in\frak a^*$ is a linear form such that $\varphi|\frak a^+>0,$ we define the \emph{entropy of $\rho$ relative to} $\varphi$ by  $$h_\varphi=\lim_{s\to\infty}\frac1t\log\#\{[\g]\in[\G]\textrm{ non-torsion}:\varphi(\lambda(\rho\g))\leq t\}.$$ 

\begin{prop}[{\cite[Section 7]{quantitative}}]\label{prop:finito2} Let $\rho:\G\to G$ a Zariski-dense hyperconvex representation, and consider  $\varphi\in\frak a^*$ such that $\varphi|\frak a^+-\{0\}>0,$ then $h_\varphi\in(0,\infty).$
\end{prop}

The \emph{barycenter} of the Weyl chamber $\frak a^+$ is the half line contained in $\frak a^+$ determined by $$\sf{bar}_{\frak a^+}=\{a\in\frak a^+:\a_1(a)=\a_2(a)\textrm{ for every pair }\a_1,\a_2\in\Pi\}.$$

\begin{teoC} Let $\rho:\G\to G$ be a Zariski-dense hyperconvex representation, and $\varphi\in\frak a^*$ a linear form such that $\varphi|\frak a^+-\{0\}>0,$ then $$\alpha h_\varphi\leq \hX\frac {\a(\sf{bar}_{\frak a^+})}{\varphi(\sf{bar}_{\frak a^+})},$$ where $\a\in\Pi$ is any simple root and $\z$ is $\alpha$-H\"older.
\end{teoC}

Note that the direction of $\frak a^+$ that gives the upper bound, does not depend on the linear form $\varphi.$

Denote by $$\alpha_\rho=\sup\{\alpha\in\R_+^*:\z:\Lim_\G\to\scr F\textrm{ is $\alpha$-H\"older}\}.$$

\begin{teoD}Let $\G$ be a Zariski-dense convex cocompact group of $\isom_+\H^k,$ and consider a Zariski-dense hyperconvex representation $\rho:\G\to G.$ Assume there exists $(\varphi\in\frak a^+)^*$ such that $$\alpha_\rho h_\varphi=\hX\frac {\a(\sf{bar}_{\frak a^+})}{\varphi(\sf{bar}_{\frak a^+})},$$ where $\a\in\Pi$ is any simple root. Then $\rho$ extends as an isomorphism $\vo\rho:\isom_+\H^k\to G.$
\end{teoD}

Theorem D together with a theorem of Guichard (\ref{Zariski-Guichard} below), give the following corollary whose proof is postponed to the end of this article. Recall that $\E$ is a closed oriented hyperbolic surface.

\begin{cor}\label{cors:hitchincor} Let ${\sf{f}}:\pi_1\E\to\PSL(2,\R)$ be a hyperbolization of $\E,$ and consider a representation in the Hitchin component $\rho:\pi_1\E\to\PGL(d,\R).$ Then $$\alpha_\rho h_\rho \leq \frac2{d-1}\textrm{ and }\alpha_\rho \h_\rho \leq \frac2{d-1},$$ and either equality holds only if $\rho=\tau_d\circ {\sf{f}},$ where $\tau_d:\PSL(2,\R)\to\PSL(d,\R)$ is the irreducible representation.
\end{cor}

Remark that if $\z:\bord\pi_1\E\to\scr F$ is the equivariant map of a Hitchin representation then, by definition, it is less (or equally) regular than $\xi=\z_1:\bord\pi_1\E\to\P(\R^d).$ Hence, even though we obtain a much better bound on $\alpha_\rho h_\rho,$ we do not know if this is produced by a decay of regularity of the map $\z.$

\subsubsection*{Acknowledgments} The author is extremely thankful to Martin Bridgeman, Dick Canary, Francois Labourie, Alejandro Passeggi and Rafael Potrie for useful discussions. He would like to particularly thank Yves Benoist, Matias Carrasco and Jean-Fran\c cois Quint for discussions that considerably improved the statements of this work, and Qiongling Li for pointing out an error on the first version of this paper.

% !TEX root = regularityandrigidity.tex

\section{Reparametrizations and Thermodynamic Formalism}

Let $X$ be a compact metric space and let $\phi=(\phi_t:X\to X)_{t\in\R}$ be a continuous flow on $X$ without fixed points. Consider a positive continuous function $f:X\to\R_+^*$ and define $\k:X\times\R\to\R$ by \begin{equation}\label{equation:k} \k(x,t)=\int_0^tf\phi_s(x)ds.\end{equation} %if $t$ is positive, and $\k(x,t):=-\k(\phi_tx,-t)$ for $t$ negative. 
The function $\k$ has the cocycle property $\k(x,t+s)=\k(\phi_t x,s)+\k(x,t)$ for every $t,s\in\R$ and $x\in X.$

Since $f>0$ and $X$ is compact, $f$ has a positive minimum and $\k(x,\cdot)$ is an increasing homeomorphism of $\R.$ We then have a map $\alpha:X\times\R\to\R$ such that \begin{equation}\label{equation:inversa} \alpha(x,\k(x,t))=\k(x,\alpha(x,t))=t,\end{equation} for every $(x,t)\in X\times\R.$

\begin{defi}\label{defi:repa}The \emph{reparametrization} of $\phi$ by $f$ is the flow $\psi=\psi^f=\{\psi_t:X\to X\}_{t\in\R}$ defined by $\psi_t(x)=\phi_{\alpha(x,t)}(x).$ If $f$ is H\"older-continuous we will say that $\psi$ is a H\"older reparametrization of $\phi.$
\end{defi}

A function $U:X\to\R$ is $\clase^1$ \emph{in the direction of the flow} $\phi,$ if for every $p\in X$ the function $t\mapsto U(\phi_t(p))$ is of class $\clase^1,$ and the function $$p\mapsto \left.\frac{\partial }{\partial t}\right|_{t=0}U(\phi_t(p))$$ is continuous. Two H\"older-continuous functions $f,g:X\to\R$ are \emph{Liv\v sic-co\-ho\-mo\-lo\-gous,} if there exists a continuous function $U:X\to\R,$ $\clase^1$ in the direction of the flow, such that for all $p\in X$ one has $$f(p)-g(p)=\left.\frac{\partial}{\partial t}\right|_{t=0} U(\phi_t(p)).$$

\begin{obs} If $f,g:X\to\R_+^*$ are continuous and Liv\v sic-cohomo\-lo\-gous, the reparametrization of $\phi$ by $f$ is conjugated to the reparametrization by $g,$ i.e. there exists a homeomorphism $h:X\to X$ such that for all $p\in X$ and $t\in\R$ $$h(\psi_t^fp) =\psi_t^g(hp).$$
\end{obs}

Let $\psi$ be the reparametrization of $\phi$ by $f:X\to\R^*_+.$ If $\tau$ is a periodic orbit of $\phi$ of period $p(\tau),$ then the period of $\tau$ for $\psi$ is \begin{equation}\label{equation:periodos}\int_\tau f=\int_0^{p(\tau)}f(\phi_s(x))ds,\end{equation} where $x\in\tau.$ If $m$ is a $\phi$-invariant probability measure on $X,$ the probability measure $m^\#$ defined by $$\frac{dm^\#}{dm}(\cdot)=f(\cdot)/\int fdm,$$ is $\psi$-invariant. This relation between invariant probability measures induces a bijection and Abramov \cite{abramov} relates the corresponding metric entropies: \begin{equation}\label{eq:abramov}h(\psi,m^\#)=h(\phi,m)/\int fdm.\end{equation}

Denote by $\cal M^{\phi}$ the set of $\phi$-invariant probability measures. The \emph{pressure} of a continuous function $f:X\to\R$ is defined by $$P(\phi,f)=\sup_{m\in\cal M^{\phi}}h(\phi,m)+\int_X fdm.$$ A probability $m$ such that the supremum is attained is called an \emph{equilibrium state} of $f.$ An equilibrium state for $f\equiv0$ is called a \emph{probability of maximal entropy}, its entropy is called the \emph{topological entropy} of $\phi$ and denoted $h_{\textrm{top}}(\phi).$

\begin{lema}[{\cite[Section 2]{quantitative}}]\label{lema:entropia2} Let $\psi$ be the reparametrization of $\phi$ by $f:X\to\R_+^*,$ and assume that $h_{\tope}(\psi)$ is finite. Then $m\mapsto m^\#$ induces a bijection between the set of equilibrium states of $-h_{\tope}(\psi)f$ and the set of probability measures of maximal entropy of $\psi.$
\end{lema}

\subsection{Metric Anosov flows} 

We will now define \emph{metric Anosov flows}, the transfer of classical results from axiom A flows to this more general setting, is provided by Pollicott's work \cite{smaleflows}, and references therein. 

As before $\phi$ denotes a continuous flow on the compact metric space $X.$ For $\eps>0$ one defines the \emph{local stable set} of $x$ by $$W^s_\eps(x)=\{y\in X:d(\phi_tx,\phi_t y)\leq\eps\ \forall t>0\textrm{ and }d(\phi_tx,\phi_t y)\to0\textrm{ as }t\to\infty\}$$ and the local unstable set by $$W^u_\eps(x)=\{y\in X:d(\phi_{-t}x,\phi_{-t} y)\leq\eps\ \forall t>0\textrm{ and }d(\phi_{-t}x,\phi_{-t} y)\to0\textrm{ as }t\to\infty\}.$$

\begin{defi}\label{defi:anosovtop} We will say that $\phi$ is a \emph{metric Anosov flow} if the following holds: 
\begin{itemize} \item[-] There exist positive constants $C,$ $\lambda$ and $\eps$ such that for every $x\in X,$ every $y\in W^s_\eps(x)$ and every $t>0$ one has $$d(\phi_t(x),\phi_t(y))\leq Ce^{-\lambda t}$$ and such that for every $y\in W^u_\eps(x)$ one has $$d(\phi_{-t}(x),\phi_{-t}(y))\leq Ce^{-\lambda t}.$$
\item[-] There exists a continuous map $\nu:\{(x,y)\in X\times X:d(x,y)<\delta\}\to \R$ such that $\nu(x,y)$ is the unique value such that $W^u_\eps(\phi_\nu x)\cap W^s_\eps(y)$ is non empty, and consists of exactly one point.
\end{itemize}
\end{defi}

%As before, $\phi$ denotes a flow on a compact metric space $X.$ Assume there exists a constant $\eps>0$ and a continuous map $\<,\>:\{(x,y)\in X\times X: d(x,y)<\eps\}\to X$ such that \begin{itemize}\item[-]$\<x,x\>=x,$\item[-]$\<\<x,y\>,z\>=\<x,z\>,$\item[-]$\<x,\<y,z\>\>=\<x,z\>.$\end{itemize}

%Define then $V_\delta^+(x)=\{y:y=\<y,x\>\textrm{ and }d(x,y)<\delta\}$ and $V_\delta^-(x)=\{y:y=\<x,y\>\textrm{ and }d(x,y)<\delta\}.$ If $\delta$ is small enough then $\<,\>:V_\delta^+(x)\times V_\delta^-(x)\to X$ is a homeomorphism onto an open set.

%For $C,$ $\lambda>0$ define $$W^s_\delta(x)=\{y\in V^+_\delta(x):\forall t\geq\textrm{ one has }d(\phi_t(x),\phi_t(y))\leq Ce^{-\lambda t}\}$$ and $$W^u_\delta(x)=\{y\in V^-_\delta(x):\forall t\geq\textrm{ one has }d(\phi_{-t}(x),\phi_{-t}(y))\leq Ce^{-\lambda t}\}.$$

A flow is said to be \emph{transitive} if it has a dense orbit. Anosov's closing Lemma is a standard dynamical tool in hyperbolic dynamics, see Sigmund \cite{sigmund}.

\begin{teo}[Anosov's closing Lemma]\label{teo:anosovclosing} Let $\phi$ be transitive metric Anosov flow, then periodic orbits are dense in $\cal M^\phi.$
\end{teo}

The following is standard in the study of Ergodic Theory of Anosov flows.

\begin{prop}[Bowen-Ruelle \cite{bowenruelle}]\label{prop:ruellebowen} Let $\phi$ be a transitive metric Anosov flow. Then given a H\"older-continuous function $f:X\to\R$ there exists a unique equilibrium state for $f.$ If two functions have the same equilibrium state, their difference is Liv\v sic-cohomologous to a constant.
\end{prop}

We will need the following immediate lemma.

\begin{lema}\label{lema:funcpositiva} Let $\phi$ be a metric Anosov flow on $X$ and let $f:X\to\R_+^*$ be H\"older-continuous. Denote by $$h_f=\lim_{t\to\infty}\frac 1t\log\#\{\tau\textrm{ periodic}:\int_\tau f\leq t\},$$ then $$\frac{h(\phi,m_{-h_f f})}{h_f}=\int fdm_{-h_ff}.$$
\end{lema}

\begin{proof} Let $\psi$ be the reparametrization of $\phi$ by $f.$ The flow $\psi$ is still a metric Anosov flow and hence its topological entropy is the exponential growth rate of its periodic orbits, i.e. the metric entropy of $\psi$ is $h_f$ (recall equation (\ref{equation:periodos})). The proof is completed by applying Lemma \ref{lema:entropia2} and Abramov's formula (\ref{eq:abramov}).
\end{proof}

%\begin{cor}[\S 3 of A.S.\cite{quantitative}]\label{cor:cociclo} Let $c:\G\times\bord\G\to\R$ be a H\"older cocycle with positve periods and finite entropy, then there exists a positive H\"older potential $f:\sf T^1M\to\R_+$ such that for every $\g\in\G$ one has $$c(\g,\g_+)=\int_{[\g]}f.$$
%\end{cor}

\section{$\CAT(-1)$ spaces}

The standard reference for this section is Bourdon \cite{bourdon}. Consider a $\CAT(-1)$ space $X,$ and $\bord X$ its visual boundary. The \emph{Busseman function} of $X,$ $B:\bord X\times\ X\times X\to\R,$ is defined by $$B(z,p,q)= B_z(p,q)= \lim_{s\to\infty} d_X(p,\sigma(s))-d_X(q,\sigma(s)),$$ where $\sigma:[0,\infty)\to X$ is any geodesic ray such that $\sigma(\infty)=z.$

Denote by $$\bord^{(2)}X=\bord X\times\bord X-\{(x,x):x\in\bord X\},$$ and fix a point $o\in X.$ The \emph{Gromov product} of $X$ based on $o,$ $[\cdot,\cdot]_o:\bord^{(2)} X\to\R,$ is defined by $$[x,y]_o=\frac12(B_x(o,p)+B_y(o,p))$$ where $p$ is any point in the geodesic joining $x$ and $y.$ Remark that $[x,y]_o\to\infty$ as $y$ approaches $x.$ The \emph{visual metric} on $\bord X$ based on $o,$ is defined by $\d_o(x,y)=e^{-[x,y]_o}.$ Since $X$ is $\CAT(-1)$ this is in fact a distance on $\bord X.$

For $\g\in\isom X,$ denote by $|\g|$ its translation length, $|\g|=\inf_{p\in X}d_X(p,\g p).$ If $\g$ is hyperbolic then one has $$|\g|=B_{\g_+}(\g^{-1}o,o),$$ for any $o\in X,$ where $\g_+$ is the attractor of $\g$ on $\bord X.$

%The following Lemma proves that the translation length of a hyperbolic element, can be computed from its action on $\bord X.$

\begin{lema}\label{lema:lambda2hyp} Consider a hyperbolic element $\g\in\isom X,$ then for any $x\in\bord X-\{\g_-\}$ one has $$\lim_{n\to\infty}\frac{\log
\d_o(\g^nx,\g_+)}n=-|\g|.$$
\end{lema}

\begin{proof} This is standard (Yue \cite{yue}). Fix two points $x,z\in\bord X,$ then for every $\g\in\isom X,$ one has $$\d_o(\g z,\g x)=e^{\frac{1}2 (B_{\g z}(\g o,o)+B_{\g x}(\g o,o))} \d_o(z,x).$$

Hence, for a given $\eps$ there exists a neighborhood $V$ of $z,$ such that for every $x\in V$ one has  $$1-\eps\leq \frac{\d_0(\g z,\g x)}{\d_o(z,x)} e^{-B_{\g z}(\g o,o)}\leq 1+\eps.$$

Assume now that $\g$ is hyperbolic, consider $z=\g_+,$ and choose $V$ with the additional property $\g V\subset V.$ Fix $\eps>0$ and assume that $x\in V,$ then one has $$(1-\eps)^n\leq \frac{\d_o(\g_+,\g^n x)}{\d_o(\g_+,x)} e^{-nB_{\g_+}(\g o,o)}\leq (1+\eps)^n.$$ Taking logarithm and dividing by $n$ one obtains the desired conclusion. If $x\notin V,$ then a big enough power $\g^N x$ does lie in $V$ (recall $x\neq\g_-$), and one repeats the argument.
\end{proof}

For a discrete subgroup $\G$ of $\isom X,$ denote by $\Lim_\G$ its limit set on $\bord X.$  Consider the space $\w{\UG}$ defined by $$\{\sigma:(-\infty,\infty)\to X: \text{$\sigma$ is a complete geodesic with $\sigma(-\infty),\sigma(\infty)\in\Lim_\G$}\}.$$ 

The group $\G$ naturally acts on $\w{\UG}$ and we denote by $\UG = \G\/\w{\UG}$ its quotient. We will say that $\G$ is \emph{convex cocompact} if the space $\UG$ is compact. 

\begin{obs} Throughout this work we will fix a convex cocompact action of $\G$ on $X,$ hence we allow ourselves to naturally identify $\Lim_\G$ to $\bord\G,$ and to refer to the space $\UG$ as only depending on $\G.$
\end{obs}

The space $\UG$ is naturally equipped with a flow $\phi=\{\phi_t:\UG\to\UG\}_{t\in\R}$ simply by changing the parametrization of a given complete geodesic. This is called \emph{the geodesic flow of $\G$}

The following theorem relates this section to the preceding one:

\begin{teo}[c.f. Bourdon \cite{bourdon}]\label{teo:flujo} Let $\G$ be a convex cocompact group of $\isom X.$ Then the geodesic flow of $\G$ is a metric Anosov flow. The topological entropy of the geodesic flow is hence $$\hX=\lim_{t\to\infty}\frac1t\log\#\{[\g]\in[\G]\textrm{ non-torsion}:|\g|\leq t\}.$$
\end{teo}

\begin{prop}[c.f. Bourdon \cite{bourdon}] Consider a convex cocompact group $\G$ of $\isom X$ and $\rho:\G\to\isom Y$ a convex cocompact action on a $\CAT(-1)$ space $Y.$ Then there exists a H\"older-continuous equivariant map $\xi:\Lim_\G\to\Lim_{\rho\G}.$
\end{prop}

The regularity of the equivariant map is directly related to the ratios of the periods:

\begin{lema}\label{lema:alpha} Consider a convex cocompact group $\G$ of $X$ and $\rho:\G\to\isom Y$ a convex cocompact action on a $\CAT(-1)$ space $Y.$ Then for every non torsion $\g\in\G,$ one has $$\alpha\leq\frac{|\rho\g|}{|\g|},$$ where $\xi$ is $\alpha$-H\"older. 
\end{lema}

\begin{proof}
Consider a non torsion $\g\in\G.$ Lemma \ref{lema:lambda2hyp} states that, for any $x\in\bord X-\{\g_-\},$ one has $$ |\rho\g|=\lim_{n\to\infty}\frac{\log d(\rho\g^n(\xi x),(\rho\g)_+)}n=\lim_{n\to\infty}\frac{\log d(\xi(\g^nx),\xi(\g_+))}n,$$ since $\xi$ is equivariant. H\"older continuity of $\xi$ implies that the last quantity is bounded above by $$ \lim_{n\to\infty}\frac{\log K\d_o(\g^n x,\g_+)^\alpha}n=-\alpha|\g|,$$ again using Lemma \ref{lema:lambda2hyp}. Thus, for every non torsion $\g\in\G,$ one has $$\alpha \leq \frac{|\rho\g|}{|\g|}.$$ This finishes the proof.
\end{proof}

\subsection{H\"older cocycles} We will now focus on H\"older cocycles on $\bord\G.$ The main references for this subsection are Ledrappier \cite{ledrappier} and \cite[Section 5]{quantitative}.

\begin{defi}\label{defi:cociclo}A \emph{H\"older cocycle} is a function $c:\G\times\bord\G\to\R$ such that $$c(\g_0\g_1,x)=c(\g_0,\g_1x)+c(\g_1,x)$$ for any $\g_0,\g_1\in\G$ and $x\in\bord\G,$ and where $c(\g,\cdot)$ is a H\"older map for every $\g\in\G$ (the same exponent is assumed for every $\g\in\G$). 
\end{defi}

Given a H\"older cocycle $c$ and $\g\in\G-\{e\},$ the \emph{period} of $\g$ for $c$ is defined by $$\l_c(\g)=c(\g,\g_+),$$ where $\g_+$ is the attractive fixed point of $\g$ on $\bord\G.$ The cocycle property implies that the period of $\g,$ only depends on its conjugacy class $[\g]\in[\G].$

Two H\"older cocycles $c,c':\G\times\bord\G\to\R$ are \emph{cohomologous}, if there exists a H\"older-continuous function $U:\bord\G\to\R,$ such that for all $\g\in\G$ one has $$c(\g,x)-c'(\g,x)=U(\g x)-U(x).$$ One easily deduces from the definition that the set of periods $\{\l_c(\g):\g\in\G\textrm{ non torsion}\},$ of a H\"older cocycle, is a cohomological invariant.

\begin{teo}[Ledrappier \cite{ledrappier}]\label{teo:ledrappier} Two H\"older cocycles are cohomologous if and only if their periods coincide for every non-torsion $\g\in\G.$ For a given H\"older cocycle c there exists a H\"older-continuous function $f_c:\UG\to\R,$ such that for every non-torsion $[\g]$ one has $$\int_{[\g]}f_c=\l_c(\g).$$ If $c$ is cohomologous to $c'$ then $f_c$ is Liv\v sic-cohomologous to $f_{c'}.$
\end{teo}

We are interested in cocycles whose periods are non-negative, i.e. such that $\l_c(\g)\geq0$ for every non torsion $\g\in\G.$ The \emph{entropy}\footnote[2]{In \cite{quantitative} this is called the exponential growth rate of the cocycle.} of such cocycle is defined by $$h_c=\limsup_{t\to\infty}\frac 1t\log\#\{[\g]\in[\G]\textrm{ non-torsion}:\l_c(\g)\leq t\}\in\R_+\cup\{\infty\}.$$

The Busseman function induces a H\"older cocycle on $\bord\G$ as follows. Fix a point $o\in X,$ consider $\xi:\bord\G\to\Lim_\G$ the equivariant map, and define $\bus_\G:\G\times\bord\G\to\R$ by $$\bus_\G(\g,x)=B_{\xi(x)}(\g^{-1}o,o).$$ The period $\bus_\G(\g,\g_+)=|\g|,$ is the length of the closed geodesic associated to $\g,$ and the entropy of $\bus_\G$ is $\hX.$

\begin{lema}[{\cite[Section 3]{quantitative}}]\label{lema:cocicloss} Let $c$ be a H\"older cocycle with $h_c\in(0,\infty),$ then $f_c$ is Liv\v sic-cohomologous to a positive function. 
\end{lema}

\begin{lema}\label{lema:intersection} Consider a H\"older cocycle $c$ with finite and positive entropy. Then there exists a positive number $\LL(c),$ and a sequence $\g_n\to\infty$ in $\G,$ such that $$\frac{\ell_c(\g_n)}{|\g_n|}\to\LL(c)\leq\frac{\hX}{h_c},$$ as $n\to\infty.$ Moreover, if $\LL(c)=\hX/h_c,$ then there exists a constant $\kappa>0,$ such that $c$ and $\kappa\bus_\G$ are cohomologous.
\end{lema}

In the language of \cite{presion}, one has $\LL(c)\II(f_c,1)=1,$ and the lemma is direct consequence of \cite[Proposition 7.7]{presion}. Nevertheless, we give a proof for completeness.

\begin{proof}Applying Lemma \ref{lema:cocicloss}, there exists a positive, H\"older-continuous function $f_c:\UG\to\R_+^*,$ such that, for every non torsion conjugacy class $[\g]$ of $[\G],$ one has $$\int_{[\g]}f_c= \l_c(\g).$$ Denote by $m_{-h_c f_c}$ the equilibrium state of $-h_c f_c,$ and consider a sequence of periodic orbits $\{[\g_n]\},$ such that $$\frac{\Leb_{\g_n}}{|\g_n|}\to m_{-h_c f_c},$$ as $n\to\infty.$ The existence of this sequence is guaranteed by Anosov's closing Lemma \ref{teo:anosovclosing}. Thus, $$\frac{\l_c(\g_n)}{|\g_n|}=\frac1{|\g_n|}\int_{[\g_n]} f\to \int f dm_{-h_c f_c}$$ which, using Lemma \ref{lema:funcpositiva}, is equal to  $$\frac{h(\phi,m_{-h_c f_c})}{h_c}.$$ Define $\LL(c)=h(\phi,m_{-h_c f_c})/h_c.$

Recall that $\hX$ is the maximal entropy of $\phi,$ hence $\LL(c)\leq \hX/h_c,$ and the equality $\LL(c)=\hX/h_c$ implies that $m_{-h_\rho f_c}$ is the measure of maximal entropy of $\phi.$ Thus, Proposition \ref{prop:ruellebowen} implies that the function $f_c$ is Liv\v sic-cohomologous to a constant and the proof is completed.
\end{proof}

If $\rho:\G\to\isom(Y)$ is a convex cocompact action on a $\CAT(-1)$ space $Y,$ denote by $$\alpha_\rho=\sup\{\alpha\in\R_+^*:\textrm{the equivariant map $\xi:\Lim_\G\to\Lim_{\rho\G}$ is $\alpha$-H\"older}\}.$$ We can now prove the following proposition stated in the Introduction, this is a simpler version of the arguments for Theorem A.

\begin{prop} Consider a convex cocompact group $\G$ of $X$ and consider a convex cocompact action $\rho:\G\to\isom(Y),$ where $Y$ is $\CAT(-1),$ such that $\alpha_{\rho} h_{\rho}=\hX.$ Then the H\"older cocycles $\bus_{\rho\G}$ and $\alpha_\rho\bus_\G$ are cohomologous.
\end{prop}

\begin{proof} Recall that $h_{\rho}$ is the entropy of the H\"older cocycle $\bus_{\rho\G},$ hence $h_{\rho}\in(0,\infty).$ Applying Lemma \ref{lema:intersection} to the cocycle $\bus_{\rho\G}$ one obtains a sequence $\{\g_n\}$ in $\G$ such that $$\frac{\ell_c(\g_n)}{|\g_n|}\to\LL(\bus_{\rho\G})\leq\frac{\hX}{h_\rho}.$$ Using Lemma \ref{lema:alpha} one has $$\alpha_\rho\leq\frac{|\rho\g_n|}{|\g_n|}\leq \LL(\bus_{\rho\G})(1+\eps)\leq\frac{\hX}{h_\rho}(1+\eps),$$ for a given $\eps>0$ and big enough $n.$ The equality $\alpha_{\rho} h_{\rho}=\hX$ implies $\LL(\bus_{\rho\G})=\hX/h_\rho$ and hence there exists $\kappa$ such that $\bus_{\rho\G}$ and $\kappa\bus_\G$ are cohomologous. Again $\alpha_{\rho} h_{\rho}=\hX$ implies $\kappa=\alpha_{\rho}.$
\end{proof}

\section{Convex representations}\label{section:convexas}

Let $\G$ be a convex cocompact isometry group of a $\CAT(-1)$ space.

\begin{defi}A representation $\rho:\G\to\PGL(d,\R)$ is \emph{convex} if there exists a $\rho$-equivariant H\"older-continuous map $$(\xi,\xi^*):\bord\G\to\P(\R^d) \times\P((\R^d)^*),$$ such that $\R^d=\xi(x)\oplus\ker\xi^*(y)$ whenever $x\neq y.$
\end{defi}

\begin{lema}\label{lema:irreducible} Let $\rho:\G\to\PGL(d,\R)$ be a convex representation, then the action of $\rho\G$ on $\<\xi(\bord\G)\>$ is irreducible.
\end{lema}

\begin{proof} Consider $W\subset\<\xi(\bord\G)\>$ a $\rho\G$-invariant subspace. Consider $w\in W$ and write $$w=\sum_{i=1}^k \alpha_iv_i$$ where $v_i\in\xi(x_i)$ for $k$-points $x_i\in\bord\G.$ Consider now some non torsion $\g\in\G$ such that $\g_-\notin\{\upla x1k\}.$ We then have $\g^n x_i\to\g_+$ and hence  $\R \rho\g^n(w)\to\xi(\g_+)$ in $\P(\R^d).$ Thus $\xi(\g_+)\in W,$ since $W$ is $\rho\G$-invariant one has $$\xi(\bord\G)=\xi(\overline{\G\cdot\g_+})\subset W.$$ This finishes the proof.

\end{proof}

We say that $g\in\PGL(d,\R)$ is \emph{proximal}, if it has a unique complex eigenvalue of maximal modulus, and its generalized eigenspace is one dimensional. This eigenvalue is necessarily real, and its modulus is equal to $\exp\lambda_1(g).$ Denote by $g_+,$ the $g$-fixed line of $\R^d$ consisting of eigenvectors of this eigenvalue, and $g_-$ the $g$-invariant complement of $g_+$ (i.e. $\R^d=g_+\oplus g_-$). The line $g_+$ is an attractor on $\P(\R^d),$ for the action of $g,$ and $g_-$ is the repelling hyperplane.

\begin{lema}[{\cite[Section 3]{quantitative}}]\label{lema:loxodromic} Let $\rho:\G\to \PGL(d,\R)$ be a convex irreducible representation. Then for every non torsion element $\g\in\G,$ $\rho (\g)$ is proximal, $\xi(\g_+)$ is its attractive fixed line and $\xi^*(\g_-)$ is the repelling hyperplane. Consequently $\xi(x)\subset\xi^*(x)$ for every $x\in \bord\G.$
\end{lema}

Fix now a norm $\|\ \|$ on $\R^d.$ We define the H\"older cocycles $\beta_\rho,\vo\beta_\rho:\G\times\bord\G\to\R$ by $$\beta_\rho(\g,x)=\log\frac{\|\rho(\g)v\|}{\|v\|}\textrm{ and }
\vo \beta_\rho(\g,x)=\log\frac{\|\theta\circ\rho(\g^{-1})\|}{\|\theta\|},$$ for a non zero $v\in\xi(x),$  and a non zero linear form $\theta\in\xi^*(x).$ Lemma \ref{lema:loxodromic} implies the following.

\begin{lema}[{\cite[Section 3]{quantitative}}]\label{lema:periodo} Assume $\rho$ is convex and irreducible, then for every non-torsion $\g\in\G$ one has $\ell_{\beta_\rho}(\g)=\lambda_1(\rho\g),$ and $\ell_{\vo\beta_\rho}(\g)=\lambda_1(\rho\g^{-1})=-\lambda_d(\rho\g).$
\end{lema}

\subsection{Adjoint representation} Given an irreducible convex representation $\rho:\G\to\PGL(d,\R)$ we will now show how the Adjoint representation $\Ad:\PGL(d,\R)\to\PGL(\frak{sl}(d,\R))$ induces again an irreducible convex representation $\A_\rho$ such that $$\lambda_1(\A_\rho\g) =\lambda_1(\rho\g) -\lambda_d(\rho\g).$$ This is standard.

Recall that the \emph{adjoint representation} is defined by conjugation $\Ad(g)(T)=gTg^{-1},$ where $T\in\frak{sl}(d,\R)=\{\textrm{traceless endomorphisms of $\R^d$}\}.$ Consider $\scr  F_*(\R^d)$ the space of incomplete flags consisting of a line contained on a hyperplane, $$\scr F_*(\R^d)=\{(v,\theta)\in\P(\R^d)\times\P((\R^d)^*): \theta(v)=0\}.$$

Given $(v,\theta)\in\scr F_*$ define $M(v,\theta)\in\P(\frak{sl}(d,\R))$ by $M(v,\theta)(w)=\theta(w)v$ and define $\Phi(v,\theta)\in\P(\frak{sl}(d,\R)^*)$ by $\Phi(v,\theta)(T)=\theta(Tv).$ These maps induce a map $$(M,\Phi):\scr F_*(\R^d)\to\scr F_*(\frak{sl}(d,\R)).$$

Say that two points $(v,\t),(w,\varphi)\in\scr F_*(\R^d)$ are \emph{in general position} if $$\theta(w)\neq0\textrm{ and }\varphi(v)\neq 0.$$

\begin{lema}\label{lema:aldope} The maps $M$ and $\Phi$ are $\Ad$-equivariant. If $(v,\t),(w,\varphi)\in\scr F_*(\R^d)$ are in general position, the points $$(M,\Phi)(v,\theta)\textrm{ and }(M,\Phi)(w,\varphi)$$ are also in general position. If $g$ and $g^{-1}$ are proximal then $\Ad g$ is proximal and its attractor  is $M(g_+,{(g^{-1})}_-).$
\end{lema}

The proof of the lemma is standard and direct.

\begin{lema}\label{lema:adjunta} Consider a convex irreducible representation $\rho:\G\to\PGL(d,\R)$ and consider the map $\eta=M\circ(\xi,\xi^*):\bord\G\to\P(\frak{sl}(d,\R)).$ Denote by $V_\rho=\<\eta(\bord\G)\>$ and $$\eta^*=(\Phi\circ(\xi,\xi^*))\cap V_\rho.$$ Then $\A_\rho=\Ad\circ\rho|V_\rho:\G\to\PGL(V_\rho)$ is an irreducible convex representation with equivariant maps $(\eta,\eta^*),$ moreover for a non torsion $\g\in\G,$ one has $$\lambda_1(\A_\rho\g)=\lambda_1(\rho\g)-\lambda_d(\rho\g).$$

\end{lema}

We will say that $\A_\rho$ is the \emph{irreducible adjoint representation of $\rho.$}

\begin{proof} Irreducibility follows from Lemma \ref{lema:irreducible}. The other properties are consequence of Lemma \ref{lema:aldope}, together with Lemma \ref{lema:loxodromic}. The last statement follows from the fact that, if $\g\in\G$ is non torsion, then $\xi^*(\g_+)$ is the repelling hyperplane of $\rho\g^{-1}$ and hence $$M(\xi(\g_+),\xi^*(\g_+)),$$ the attractor of $\A_\rho\g,$ belongs to $V_\rho.$
\end{proof}

\subsection{Regularity}

The following lemma is from Benoist \cite{convexes1}.

\begin{lema}[Benoist \cite{convexes1}]\label{lema:lambda2} Let $g\in\PGL(V)$ be proximal and let $V_{\lambda_2(g)}$ be the sum of the  characteristic spaces of $g$ whose associated eigenvalue is of modulus $\exp\lambda_2(g).$ Then for every $v\notin\P(g_-)$ with non zero component in $V_{\lambda_2(g)}$ one has $$\lim_{n\to\infty}\frac{\log d_\P(g^n(v),g_+)}n= \lambda_2(g)-\lambda_1(g).$$
\end{lema}

The following lemma relates the H\"older exponent of the equivariant map, and eigenvalues of $\rho(\g)$ for non-torsion $\g\in\G.$

\begin{lema}\label{lema:1} Let $\rho:\G\to\PGL(d,\R)$ be a convex irreducible representation then, for every non torsion $\g\in\G,$ one has $$\alpha\leq\min\left\{ \frac{\lambda_1(\rho\g)-\lambda_2(\rho\g)}{|\g|},\frac{\lambda_{d-1}(\rho\g)-\lambda_d(\rho\g)}{|\g|}\right\},$$ when $\xi$ is $\alpha$-H\"older.
\end{lema}

\begin{proof}Consider a non torsion $\g\in\G.$ Since $\rho$ is irreducible, there exists $x\in\bord\G-\{\g_-\}$ such that $\xi(x)$ has non zero projection to $V_{\lambda_2(\rho\g)},$ the characteristic space of $\rho\g$ of eigenvalue of modulus $\exp\lambda_2(\rho\g).$ Lemma \ref{lema:loxodromic} states that $\xi(\g_+)$ is the attractor of $\rho\g.$ Applying Benoist's Lemma \ref{lema:lambda2}, we obtain  $$\lambda_2(\rho\g)- \lambda_1(\rho\g)=\lim_{n\to\infty}\frac{\log d_\P(\rho\g^n(\xi x),\xi(\g_+))}n=\lim_{n\to\infty}\frac{\log d_\P(\xi(\g^nx),\xi(\g_+))}n,$$ since $\xi$ is equivariant. H\"older continuity of $\xi$ implies that the last quantity is smaller than $$ \lim_{n\to\infty}\frac{\log K\d_o(\g^n x,\g_+)^\alpha}n=-\alpha|\g|,$$ according to Lemma \ref{lema:lambda2hyp}. Thus, for every non torsion $\g\in\G,$ one has $$\alpha \leq \frac{\lambda_1(\rho\g)-\lambda_2(\rho\g)}{|\g|},$$ applying this inequality to $\g^{-1}$ one obtains $$\alpha \leq \frac{\lambda_{d-1}(\rho\g)-\lambda_d(\rho\g)}{|\g|}.$$
\end{proof}

\section{Proof of Theorem A}

This section is devoted to the proof of Theorem A. Consider an irreducible convex representation $\rho:\G\to\PGL(d,\R).$ Proposition \ref{prop:finito} states that $h_\rho\in(0,\infty).$ Since $\A_\rho$ is also convex and irreducible one gets $\h_\rho=2h_{\A_\rho}\in(0,\infty). $

Denote by $c$ either the H\"older cocyle $$\beta_\rho, \textrm{ or } \frac{\beta_\rho+\vo\beta_\rho}2.$$ Remark that, either $h_c=h_\rho,$ or  $h_c=\h_\rho.$

Using Lemma \ref{lema:intersection} for $c,$ one obtains a sequence $\{\g_n\}$ in $\G,$ such that $$\frac{\ell_c(\g_n)}{|\g_n|}\to\LL(c)\leq\frac{\hX}{h_c}.$$ Lemma \ref{lema:1} then gives \begin{equation}\label{equation:desigualdad1}\alpha\leq\min\left\{ \frac{(\lambda_1-\lambda_2)(\rho\g_n)}{|\g_n|},\frac{(\lambda_{d-1}-\lambda_d)(\rho\g_n)}{|\g_n|}\right\}\end{equation} $$\leq \min\left\{ \frac{(\lambda_1 -\lambda_2)(\rho\g_n)}{\l_c(\g_n)},\frac{(\lambda_{d-1} -\lambda_d)(\rho\g_n)}{\l_c(\g_n)} \right\} \LL(c)(1+\eps),$$ for a given $\eps$ and big enough $n.$

We will now distinguish the two cases $c=\beta_\rho$ and $c=(\beta_\rho+\vo\beta_\rho)/2$ separetly:

\subsection*{First case: $c=\beta_\rho$} In this case $\l_c(\g)=\lambda_1(\rho\g),$ $h_c=h_\rho$ (the spectral entropy of $\rho$) and equation (\ref{equation:desigualdad1}) is $$\frac{\alpha h_\rho}{\hX}\leq \frac{\alpha } {\LL(\beta_\rho)}\leq \min\left\{ \frac{\lambda_1 -\lambda_2}{\lambda_1}(\rho\g_n),\frac{\lambda_{d-1} -\lambda_d}{\lambda_1}(\rho\g_n) \right\}(1+\eps).$$ We will now maximize the function ${\sf V}_1:\P(\frak a^+)\to\R$ defined by $${\sf V}_1(a_1,\ldots,a_d)= \min\left\{\frac{a_1-a_2}{a_1},\frac{a_{d-1}-a_d}{a_1}\right\}.$$ Recall that $$\frak a^+=\{(a_1,\ldots,a_d)\in\R^d:a_1+\cdots+a_d=0\textrm{ and }a_1\geq\cdots\geq a_d\}$$ and consider $a\in\frak a^+.$ We will distinguish two cases.

\noindent
\underline{Assume $a_2\geq0$:} In this case one has $${\sf{V}_1(a)}\leq \frac{a_1-a_2}{a_1}=1-\frac{a_2}{a_1}\leq 1.$$

\noindent\underline{Assume $a_2<0$:}

\begin{lema}\label{obs:a2} In this case one has $a_1-a_2> a_{d-1}-a_d,$ hence ${\sf{V}}_1(a)=(a_{d-1}-a_d)/a_1.$
\end{lema}

\begin{proof}Recall that $a_{k+1}-a_k\leq0$ for all $k\in\{1,\ldots,d-1\}.$ Using the following tricky equality (recall $d\geq3$) $$a_1+(d-1)a_2+\sum_{k=2}^{d-1} (d-k)(a_{k+1}-a_k)=a_1+a_2+\cdots+a_d=0,$$ one obtains $$a_1-a_2+a_d-a_{d-1}=-da_2-\sum_{k=2}^{d-2} (d-k)(a_{k+1}-a_k)>0.$$ Hence $a_1-a_2>a_{d-1}-a_d.$

\end{proof}

Since $0>a_2\geq\cdots\geq a_d$ one has $$a_1=-a_2-\cdots-a_d>-(a_{d-1}+a_d)\geq0.$$ Given that  $d\geq3$ one obtains, $a_{d-1}<0<-a_{d-1}$ and subtracting $a_d$ on each side one gets $a_{d-1}-a_d<-(a_{d-1}+a_d)<a_1,$ finally $${\sf V}_1(a)=\frac{a_{d-1}-a_d}{a_1}<1.$$

In any case one obtains ${\sf V}_1\leq1.$ We then get

\begin{equation}\label{equation:desigualdad2}\frac{\alpha h_\rho}{\hX}\leq \frac{\alpha } {\LL(\beta_\rho)}\leq{\sf V}_1(\lambda(\rho\g_n))(1+\eps)\leq 1+\eps.\end{equation} Since $\eps$ is arbitrary, we obtain the desired inequality.

\subsection*{Second case: $c=(\beta_\rho+\vo\beta_\rho)/2$} In this case we have $\l_c(\g)=(\lambda_1(\rho\g)-\lambda_d(\rho\g))/2,$ $h_c=\h_\rho$ (the Hilbert entropy of $\rho$) and inequality (\ref{equation:desigualdad1}) is $$\frac{\alpha \h_\rho}{\hX}\leq \frac{\alpha} {\LL((\beta_\rho+\vo\beta_\rho)/2)}\leq \min\left\{\frac{\lambda_1-\lambda_2}{(\lambda_1-\lambda_d)/2}(\rho\g_n),\frac{\lambda_{d-1}-\lambda_d}{(\lambda_1-\lambda_d)/2}(\rho\g_n)\right\}(1+\eps),$$ for all $n$ large enough. 

We will now maximize the function ${\sf V}_2:\P(\frak a^+)\to\R$ defined by $${\sf V}_2(a_1,\ldots,a_d)= \min\left\{\frac{a_1-a_2}{(a_1-a_d)/2},\frac{a_{d-1}-a_d}{(a_1-a_d)/2}\right\}.$$

Consider $a\in\frak a^+$ such that $$x=a_1-a_2\leq a_{d-1}-a_d=y.$$ For such $a$ one has $a_2=a_1-x$ and $a_{d-1}=y+a_d.$ Since $d\geq3$ one has $a_2\geq a_{d-1}$ hence $a_1-x\geq a_d+y\geq a_d+x$ and thus $${\sf V}_2(a)=\frac{2x}{a_1-a_d}\leq 1.$$ 

If, on the opposite, one has $a\in\frak a^+$ such that $$x=a_{d-1}-a_d\leq a_1-a_2=y,$$ then, again the fact that $a_2\geq a_{d-1}$ implies $a_1-x\geq a_1-y\geq a_d+x$ and thus $${\sf V}_2(a)=\frac{2x}{a_1-a_d}\leq 1.$$ 

In any case one obtains ${\sf V}_2\leq1.$ We then get \begin{equation}\label{equation:desigualdad} \frac{\alpha \h_\rho}{\hX}\leq \frac{\alpha} {\LL((\beta_\rho+\vo\beta_\rho)/2)}\leq{\sf V}_2(\lambda(\rho\g_n))(1+\eps)\leq 1+\eps.\end{equation} Since $\eps$ is arbitrary we obtain the desired inequality. This finishes the proof. \begin{flushright}$\square$\end{flushright}

Denote by $\alpha_\rho=\sup\{\alpha\in\R_+^*:\xi\textrm{ is $\alpha$-H\"older}\}.$ From the proof one obtains the following.

\begin{prop}\label{prop:igualdad} Let $\rho:\G\to\PGL(d,\R)$ be an irreducible convex representation. \begin{itemize}\item[i)]If $\alpha_\rho h_\rho=\hX,$ then $\beta_\rho$ and $\alpha_\rho\bus_\G$ are cohomologous. \item[ii)] If $\alpha_\rho\h_\rho=\hX,$ then $\beta_\rho+\vo\beta_\rho$ and $2\alpha_\rho\bus_\G$ are cohomologous. \end{itemize}
\end{prop}

\begin{proof} Let us prove i), the other being completely analogous. If one has $\alpha_\rho h_\rho=\hX,$ then inequality (\ref{equation:desigualdad2}) implies $\LL(\beta_\rho)=\hX/h_\rho,$ and hence using Lemma \ref{lema:intersection}, there exists $\kappa>0$ such that, $\beta_\rho$ and $\kappa\bus_\G$ are cohomologous. Equality $\alpha_\rho h_\rho=\hX$ implies then $\alpha_\rho=\kappa.$
\end{proof}

\section{Proximal representations and the limit cone of Benoist}\label{section:reps}

We will freely use the notations of subsection \ref{subsection:hyper}. For an irreducible representation $\phi:G\to\PGL(d,\R),$ denote by $\chi_\phi\in\frak a^*$ its \emph{restricted highest weight}. For every $g\in G$ one has, by definition, \begin{equation}\label{poids}\lambda_1(\phi g)=\chi_\phi(\lambda(g)).
\end{equation} The representation $\phi$ is \emph{proximal} if there exists $g\in G$ such that $\phi(g)$ is a proximal matrix. One has the following standard proposition in Representation Theory.

\begin{prop}[see Benoist {\cite[Section 2.2]{convexes3}}]\label{prop:lattices} The set of restricted weights of $\frak a^*$ is in bijection with (equivalence classes of) irreducible proximal representations of $G.$
\end{prop}

%Let $W$ be the Weyl group of $\E,$ and denote by $u_0:\frak a\to \frak a$ the biggest element in $W:$ $u_0$ is the unique element in $W$ that sends $\frak a^+$ to $-\frak a^+.$  The \emph{opposition involution} $\ii:\frak a\to\frak a$ is the defined by $\ii=-u_0.$

Consider $\{\om_\a\}_{\a\in\Pi},$ the set of fundamental weights of $\Pi.$ We will need the following result of Tits \cite{tits}.
 
\begin{prop}[Tits \cite{tits}]\label{prop:titss} For each $\a\in\Pi,$ there exists a finite dimensional proximal irreducible representation $\L_\a:G\to\PGL(V_\a),$ such that the restricted highest weight $\chi_\a$ of $\L_\a$ is an integer multiple of $\om_\a.$ 
\end{prop}

We will now specialize to the group $\isom_+\H^k.$ The Cartan subspace $\frak a_{\H^k}$ is 1-dimensional and is thus identified with $\R.$ The \emph{Jordan projection} of $\g\in\isom_+\H^k$ is $$\lambda_{\H^k}(\g)=\inf_{p\in\H^k}d_{\H^k}(p,\g p),$$ which coincides with the translation length $|\g|,$ when $\g$ is a hyperbolic element.

\begin{obs}\label{obs:klein} If $\rho:\isom_+ \H^k\to\PGL(k+1,\R)$ is the Klein model of $\H^k$ and $\g\in\isom_+\H^k$ is hyperbolic then $\lambda_1(\rho\g)=|\g|$ and $\lambda_1(\Ad\rho\g)=2|\g|.$
\end{obs}

\subsection{The limit cone of Benoist}

Let $\grupo$ be a subgroup of $G.$ The \emph{limit cone} of $\grupo$ is the closed cone of $\frak a^+$ generated by $$\{\lambda(g):g\in\grupo\}$$ and we denote it by $\cone_\grupo.$ One has the following theorem of Benoist \cite{limite}.

\begin{teo}[Benoist \cite{limite}]\label{teo:conoBenoist} Let $\grupo$ be a Zariski-dense discrete subgroup of $G,$ then $\cone_\grupo$ has non-empty interior.
\end{teo}

Let $G_i$ $i=1,2$ be center free real-algebraic semisimple Lie groups without compact factors, and denote by $\frak a_{G_i}$ a Cartan subspace of $G_i.$ The main purpose of this section is the following corollary personally communicated by Quint. 

\begin{cor}[Quint]\label{cor:grosso} Let $\rho:\grupo\to G_1$ and $\eta:\grupo\to G_2$ be Zariski-dense. Assume there exist $\varphi_1\in(\frak a_{G_1}^+)^*$ and $\varphi_2\in(\frak a_{G_2}^+)^*$ such that for all $g\in\grupo$ one has $$\varphi_1(\lambda_{G_1}(\rho g))=\varphi_2(\lambda_{G_2}(\eta g)).$$ Then $\eta\circ\rho^{-1}:\grupo\to\grupo$ extends to an isomorphism $G_1\to G_2.$
\end{cor}

\begin{proof} Let $H$ be the Zariski closure of the product representation $\rho\times\eta:\grupo\to G_1\times G_2,$ defined by $g\mapsto (\rho g,\eta g).$ Since the equation \begin{equation}\label{radio}\varphi_1(\lambda_{G_1} (g_1))=\varphi_2(\lambda_{G_2} (g_2))\end{equation} holds for every pair $(g_1,g_2)\in\rho\times\eta\,(\grupo),$ Benoist's \cite{limite} Theorem \ref{teo:conoBenoist} implies that the same relation holds for every pair $(g_1,g_2)\in H.$

The group $H\cap (G_1\times\{e\})$ is a normal subgroup of $G_1,$ it is hence (up to finite index) a product of simple factors. Equation (\ref{radio}) implies that for all $(g,e)\in H\cap (G_\rho\times\{e\})$ necessarily one has $\varphi_1(\lambda_{G_1} g)=0.$ Since $\varphi_1(v)>0$ for all $v\in\frak a_{G_1}^+-\{0\},$ one has $\lambda_{G_1}(g)=0.$ This implies that $H\cap (G_1\times\{e\})$ is a normal compact subgroup of $G_1.$ Since $G_1$ does not have compact factors and is center free one concludes that $H\cap (G_\rho\times{e})=\{e\}.$

The same argument implies that $H\cap(\{e\}\times G_2)=\{e\}$ and hence $H$ is the graph of an isomorphism extending $\eta\rho^{-1}.$
\end{proof}

We will need the following lemma.

\begin{lema}[Quint]\label{lema:factorescompactos}Let $\grupo$ be a subgroup of $\PGL(d,\R)$ acting irreducibly on $\R^d$ and with a proximal element. Then the Zariski closure of $\grupo$ is a center free semisimple Lie group without compact factors.
\end{lema}

\begin{proof} Assume that $g\in\PGL(d,\R)$ commutes with all elements on $\grupo,$ and let $\g\in\grupo$ be proximal. The attractor of $\g$ is fixed by $g$ and hence $gv=av$ for some $a\in\R$ and all $v\in\g_+.$ One easily sees that if $h\in\grupo$ is another proximal element of $\grupo$ then necessarily $gw=aw$ for $w\in h_+.$ Thus, $g$ acts as an homothethy on the vector space spanned by the attracting lines of proximal elements of $\grupo.$ Since $\grupo$ acts irreducibly this vector space is $\R^d.$ The Zariski closure $G$ of $\grupo$ is hence center free. 

Since $\grupo$ acts irreducibly so does $G,$ hence $G$ is a center free reductive Lie group, i.e.  a semisimple Lie group without center. 

Let $K$ be the maximal normal connected compact subgroup of $G,$ and let $H$ be the product of the non-compact Zariski connected, simple factors of $G.$ Then $H$ and $K$ commute and $HK$ has finite index in $G.$

Consider now a proximal element $g\in G.$ Replacing $g$ by a large enough power, we can assume that $g=hk$ for some $h\in H$ and $k\in K.$ Since eigenvalues of $k$ have modulus $1$ and $k$ and $h$ commute, we conclude that $h$ is proximal. So we can assume that $g\in H.$ 

Since $g$ and $K$ commute, the attracting line of $g$ is fixed by $K,$ and, since $K$ is connected, each vector of this attracting line is fixed by $K.$ Let $W$ be the vector space of $K$-fixed vectors on $\R^d,$ then $W$ is $G$-invariant ($K$ is normal in $G$) and nonzero. Since $G$ is irreducible on obtains $W=\R^d$ and $K=\{e\}.$
\end{proof}

\section{Hyperconvex representations and Theorem C}

Recall that $\G$ is a convex cocompact isometry group of a $\CAT(-1)$ space. We will freely use the notations of section \ref{section:reps}. Let $G$ be a real non-compact semi-simple Lie group, and denote by $\scr F$ the Furstenberg boundary of the symmetric space of $G.$ The product $\scr F\times\scr F$ has a unique open $G$-orbit, denoted by $\posgen.$

\begin{defi}A representation $\rho:\G\to G$ is \emph{hyperconvex} if there exists a $\rho$-equivariant H\"older-continuous map $\z:\bord\G\to\scr F$ such that if $x\neq y$ are distinct points in $\bord\G,$ then the pair $(\z(x),\z(y))$ belongs to $\posgen.$
\end{defi}

The following lemma relates hyperconvex representations to convex ones.

\begin{lema}\label{obs:remark} If $\rho:\G\to G$ is Zariski-dense and hyperconvex, and $\L:G\to\PGL(V)$ is a finite dimensional irreducible proximal representation, then the composition $\L\circ\rho:\G\to\PGL(V)$ is irreducible and convex.
\end{lema}

\begin{proof} A proximal representation $\L:G\to\PGL(V)$ induces a  $\clase^\infty$ equivariant map $\scr F\to\P(V).$ Considering the dual representation $\L^*:G\to\PGL(V^*)$ one obtains another equivariant map $\scr F\to\PGL(V^*).$ The remainder of the statement follows directly.
\end{proof}

%Consider $\a\in\Pi$ and $\L_\a:G\to\PGL(V_\a)$ a representation given by Tits's proposition. Since $\L_\a$ is proximal one obtains an equivariant mapping $\xi_\a:\scr F\to\P(V_\a).$

%The highest weight of the dual representation $\L_\a^*:G\to\P(V_\a^*)$ is $\chi_\a\ii,$ one thus obtains an equivariant mapping $\xi^*_\a:\scr F\to\P(V_\a^*).$ Moreover, the pair $(x,y)\in\posgen$ verifies $$\xi^*_\a(x)|\xi_\a(y)\neq0.$$ One deduces the following remark.

%Consider the \emph{limit cone} $\cone_\rho$ of $\rho\G:$ this is the closed cone of $\frak a^+$ generated by $$\{\lambda(\rho\g):\g\in\G\},$$ where $\lambda:G\to\frak a$ is the Jordan projection. One has the following Theorem of Benoist \cite{limite}.

%\begin{teo}[Benoist \cite{limite}]\label{teo:limite} The limit cone of a Zariski dense subgroup of a semi simple algebraic group has nonempty interior.
%\end{teo}

We need the following theorem from \cite{quantitative}. 

\begin{teo}[{\cite[Section 7]{quantitative}}]\label{teo:cartan} Let $\rho:\G\to G$ be a Zariski-dense hyperconvex representation, then there exists a (vector valued) H\"older cocycle $\beta:\G\times\bord\G\to\frak a$ such that, for every non torsion conjugacy class $[\g]\in[\G]$ one has, $\beta(\g,\g_+)=\lambda(\rho\g).$ If $\varphi\in\frak a^*$ is such that $\varphi|\frak a^+-\{0\}>0,$ then the H\"older cocycle $\beta^\varphi=\varphi\circ\beta$ has finite and positive entropy.
\end{teo}

Assume from now on that $\rho:\G\to G$ is a Zariski-dense hyperconvex representation, and assume that $\z:\Lim_\G\to\scr F$ is $\alpha$-H\"older.

\begin{lema}\label{lema:todo} For every simple root $\a\in\Pi,$ and every non-torsion $\g\in\G,$ one has $$\alpha\leq \frac{\a(\lambda(\rho\g))}{|\g|}.$$ 
\end{lema}

\begin{proof}Let $\L_\a\circ\rho:\G\to\PGL(V_\a)$ be the irreducible convex representation given by Tits's Proposition \ref{prop:titss} and Lemma \ref{obs:remark}. One then has $$\a(\lambda(\rho\g))= \lambda_1( \L_\a\circ\rho\g)- \lambda_2(\L_\a\circ\rho\g).$$ The lemma follows from Lemma \ref{lema:1}.
\end{proof}

\subsection{Proof of Theorem C}

The proof is very similar to the proof of Theorem A. Consider the cocycle $\beta:\G\times\bord\G\to\frak a$ given by Theorem \ref{teo:cartan}, and consider $\varphi\in\frak a^*$ such that $\varphi|\frak a^+-\{0\}>0.$ Consider the H\"older cocycle $\beta^\varphi=\varphi\circ\beta.$ Theorem \ref{teo:cartan} states that $h_{\beta^\varphi}=h_\varphi$ is finite and positive. Hence, Lemma \ref{lema:intersection} applies to the cocycle $\beta^\varphi$ and one obtains a sequence $\{\g_n\}$ in $\G$ such that $$\frac{\varphi(\lambda(\rho
\g_n))}{|\g_n|}\to\LL(\beta^\varphi)\leq\frac{\hX}{h_\varphi}.$$ Analogous reasoning to Theorem A, together with Lemma \ref{lema:todo}, yields $$\frac{\alpha h_\varphi}{\hX}\leq\frac{\alpha}{\LL(\beta^\varphi)}\leq \frac{\a(\lambda(\rho\g_n))}{\varphi(\lambda(\rho\g_n))}(1+\eps),$$ for every simple root $\a\in\Pi,$ and all big enough $n.$ We now try to maximize the function $\sf V:\P(\frak a^+)\to\R$ defined by $$\sf V(a)=\min_{\a\in\Pi}\left\{\frac{\a(a)}{\varphi(a)}\right\}.$$

We need the following standard Linear Algebra lemma. Consider an $n$-dimensional vector space $W,$ a $k$-\emph{simplex} is the convex hull of $k+1$ points $\{x_0,\ldots, x_k\}$ in $W$ such that for every $i\in\{0,\ldots,k\}$ the set $\{x_0,\ldots,x_k\}-\{x_i\}$ is linearly independent.

\begin{lema}\label{lema:bari} Consider $n+1$ affine linear forms $\varphi_i:W\to\R$ on an $n$-dimensional vector space $V,$ such that $$\Delta=\bigcap_0^n \{v\in W: \varphi_i(v)\geq0\}$$ is an $n$-dimensional simplex. Then $$\max_{v\in\Delta}\min\{\varphi_i(v):i\in\{0,\ldots,n\}\},$$ is given in the point all the $\varphi_i$'s coincide, i.e. in the unique $v\in\Delta$ such that $$\varphi_0(v)= \varphi_1(v)=\cdots= \varphi_n(v).$$
\end{lema}

Fix a vector $v$ in the interior of  $\frak a^+$ such that $\varphi(v)\neq0,$ and consider the map $T:\ker\varphi\to \P(\frak a)$ defined by $w\mapsto \R(v+w).$ This map identifies $\ker\varphi$ with $\P(\frak a)-\P(\ker\varphi).$ The functions $T_\t:\ker\varphi\to\R,$ given by $$T_\a(w)=\frac{\a(w+v)}{\varphi(w+v)}=\frac{\a(v)}{\varphi(v)}+\frac{\a(w)}{\varphi(v)},$$ are affine functionals. Since $\varphi$ is positive on the Weyl chamber $\frak a^+-\{0\},$ we get that $$\Delta=T^{-1}(\P(\frak a^+))=T^{-1}(\P(\bigcap_{\a\in\Pi} \{\a\geq0\}))= \bigcap_{\a\in\Pi}\{T_\a\geq0\}$$ is a simplex of dimension $\dim \frak a-1=\dim\ker\varphi.$

Remark that ${\sf{V}}\circ T=\min\{T_\a:\a\in\Pi\}.$ Hence Lemma \ref{lema:bari} implies that the maximum of ${\sf{V}}\circ T|\Delta$ is realized where all the functions $\{T_\a:\a\in\Pi\}$ coincide, i.e. in the set $$\{a\in\frak a^+: \a_1(a)=\a_2(a)\textrm{ for every pair $\a_1,\a_2\in\Pi$}\}.$$ This is exactly the barycenter of the Weyl chamber $\sf{bar}_{\frak a^+}.$ 

Hence \begin{equation}\label{equation:igualdadC}\frac{\alpha h_\varphi}{\hX}\leq\frac{\alpha}{\LL(\beta^\varphi)}\leq {\sf{V}}(\lambda(\rho\g_n))(1+\eps)\leq \frac{\a(\sf{bar}_{\frak a^+})}{\varphi(\sf{bar}_{\frak a^+})}(1+\eps).\end{equation}

This shows the desired inequality.

\begin{obs}\label{obs:equalC} As in Theorem A, remark that equality in equation (\ref{equation:igualdadC}) implies that there exists $\kappa>0$ such that $\beta^\varphi$ and $\kappa\bus_\G$ are cohomologous.
\end{obs}

\section{Proof of rigidity statements}\label{section:Aigualdad}

Let's prove Theorem B (Corollary \ref{cor:hilbert} and Theorem D are completely analogous). Assume $\rho:\G\to\PGL(d,\R)$ is a convex representation such that $\alpha_\rho h_\rho=\hX.$ Proposition \ref{prop:igualdad} implies that for all $\g\in\G$ one has $$\lambda_1(\rho\g)=\alpha_\rho |\g|.$$ Since $\rho\G$ is irreducible and proximal, and $\G$ is Zariski-dense in $\isom_+\H^k,$ Lemma \ref{lema:factorescompactos} and Corollary \ref{cor:grosso} imply that $\rho$ extends to $\vo\rho:\isom_+\H^k\to\PGL(d,\R).$ Hence, the equivariant map $\xi$ is the restriction of the $\clase^\infty,$ $\vo\rho$-equivariant map $\vo\xi:\bord\H^k\to\P(\R^d).$ Thus, $\xi$ is Lipschitz, i.e. $\alpha_\rho=1.$ Proposition \ref{prop:lattices} together with Remark \ref{obs:klein} imply that $\vo\rho$ is the Klein model of $\H^k.$

\section{Proof of Corollary \ref{cors:hitchincor}}

We will now prove the following corollary. Recall that $\E$ is a closed oriented hyperbolic surface.

\begin{cor2} Let ${\sf{f}}:\pi_1\E\to\PSL(2,\R)$ be a hyperbolization of $\E,$ and consider a representation in the Hitchin component $\rho:\pi_1\E\to\PSL(d,\R).$ Denote by $\alpha$ the best H\"older exponent of the equivariant map $\z:\bord\H^2\to\scr F.$ Then $$\alpha h_\rho \leq \frac2{d-1}\textrm{ and }\alpha \h_\rho \leq \frac2{d-1}.$$ Either equality holds only if $\rho=\tau_d\circ {\sf{f}},$ where $\tau_d:\PSL(2,\R)\to\PSL(d,\R)$ is the irreducible representation.
\end{cor2}

\begin{proof} Denote by $G$ the Zariski closure of $\rho,$ since $G$ is a semisimple Lie group without compact factors $\rho:\pi_1\E\to G$ is again hyperconvex. Consider $\frak a$ a Cartan subspace of $\frak g,$ and let $\chi\in\frak a^*$ be the restricted highest weight of the (irreducible proximal) representation $G\subset \PSL(d,\R),$ i.e. if $g\in G$ then $\chi(\lambda(g))=\lambda_1(g).$ Denote by $\ii:\frak a\to\frak a$ the opposition involution of $\frak a$ associated to the choice of $\frak a^+.$

Remark that by definition the entropy of $\rho$ relative to $\chi$ is the spectral entropy $h_\rho=h_\chi$ of $\rho,$ and the entropy of $\rho$ relative to $$\varphi=\frac{\chi+\chi\circ\ii}2$$ is the Hilbert entropy $\h_\rho=h_\varphi$ of $\rho.$ We will prove the corollary for the spectral entropy, the other being completely analogous.

Theorem C asserts that \begin{equation}\label{equation:algo}\alpha h_\rho\leq \frac{\t({\sf{bar}}_{\frak a^+})}{\chi({\sf{bar}}_{\frak a^+})}\end{equation} for any simple root $\t\in\Pi$ of $\frak a$ and where $\sf{bar}_{\frak{a^+}}$ is the barycenter of the Weyl chamber $\frak a^+.$ Theorem D implies that equality in (\ref{equation:algo}) can only hold if $G$ is isomorphic to $\PSL(2,\R).$ 

Guichard's Theorem gives a finite list of possible groups $G,$ i.e. of possible Zariski closures of $\rho(\pi_1\E).$ We will finish with an explicit computation showing that in all possible cases one has $$\frac{\t({\sf{bar}}_{\frak a^+})}{\chi({\sf{bar}}_{\frak a^+})}=\frac2{d-1}.$$

The author would like to thank Olivier Guichard for discussions concerning his work.

\begin{teo}[Guichard \cite{clausura}] \label{Zariski-Guichard}
Let $\rho:\pi_1\E\to\SL(d,\R)$ be the lift of a representation in the Hitchin component, then the Zariski closure $\vo{\rho^Z}$ is either conjugate to $\tau_d(\SL(2,\R)),$ $\SL(d,\R)$ or conjugate to one of the following groups:

\begin{itemize}
\item[-] $\Sp(2n,\R)$ if $d=2n,$ 
\item[-] $\SO(n,n+1)$ if $d=2n+1,$
\item[-] $\Ge$ or $\SO(3,4)$ if $d=7.$
\end{itemize}
\end{teo}

For $i\in\{1,\ldots,k\}$ we will denote by $\eps_i:\R^k\to\R$ the function $$\eps_i(\upla a1k)=a_i.$$ We refer the reader to Knapp's book \cite{g2} for the standard computations of simple roots and highest weights that follow.

\subsection*{The $\tau_d(\SL(2,\R))$ and $\SL(d,\R)$ cases}

Assume first that $\rho(\pi_1\E)$ is Fuchsian, i.e. it is Zariski dense in $\tau_d(\SL(2,\R)).$ A Cartan subspace of $\frak{sl}(2,\R)$ is $\frak a=\{(a,-a):a\in\R\}$ the Weyl chamber is $\frak a^+=\{(a,-a):a\geq0\}$ with simple root $\Pi=\{2\eps_1\}.$ The highest weight of the representation $\tau_d$ is $\chi(a,-a)=(d-1)a.$ Hence $$\frac {\t(\sf{bar}_{\frak a^+})} {\chi(\sf{bar}_{\frak a^+})}=\frac{2a}{(d-1)a}=\frac2{d-1}.$$

Suppose now that $\rho(\pi_1\E)$ is Zariski dense in $\SL(d,\R).$ The Cartan subspace of $\frak{sl}(d,\R)$ is  $\frak a=\{(a_1,\ldots,a_d)\in\R^d:a_1+\cdots+a_d=0\}$ and $$\frak a^+=\{(a_1,\ldots,a_d)\in\frak a: a_1\geq\cdots\geq a_d\},$$ the simple roots are $$\Pi=\{\a_i(a_1,\ldots,a_d)=a_i-a_{i+1}:i\in\{1,\ldots,d-1\}\}$$ and the barycenter is $$\sf{bar}_{\frak a^+}=\{((d-1)t,(d-3)t,\ldots,(3-d)t,(1-d)t):t\geq0\}.$$ Hence for any  $\t\in\Pi$ one has $$\frac {\t(\sf{bar}_{\frak a^+})} {\chi(\sf{bar}_{\frak a^+})}=\frac{2t}{(d-1)t}=\frac2{d-1}.$$

\subsection*{The $\Sp(2n,\R)$ case}

Assume $d=2n$ and that the Zariski closure of $\rho(\pi_1\E)$ is $\Sp(2n,\R).$ Standard computations show that $\frak a=\R^n,$ and a Weyl chamber is $$\frak a^+=\{(\upla a1n):a_i\geq a_{i+1}\ i=1,\ldots,n-1\textrm{ and } a_n\geq0\}.$$ The set of simple roots associated to this Weyl chamber is $$\Pi=\{\eps_i-\eps_{i+1}:i=1,\ldots,n-1\}\cup\{2\eps_n\}.$$ The barycenter of the Weyl chamber is hence $$\sf{bar}_{\frak a^+}=\{((2n-1)t,(2n-3)t,\ldots,3t,t):t\geq0\}.$$ The highest weight of the representation $\Sp(2n,\R)\subset\SL(d,\R)$ is $\chi(\upla a1n)=a_1.$ Finally, for any $\t\in\Pi$ one has $$\frac{\t(\sf{bar}_{\frak a^+})}{\chi(\sf{bar}_{\frak a^+})}=\frac{2t}{(2n-1)t}=\frac2{d-1}.$$

\subsection*{The $\SO(n,n+1)$ case}

Suppose now that $d=2n+1$ and that the Zariski closure of $\rho(\pi_1\E)$ is $\SO(n,n+1).$ Standard computations show that $\frak a=\R^n,$ and a Weyl chamber is $$\frak a^+=\{(\upla a1n):a_i\geq a_{i+1}\ i=1,\ldots,n-1\textrm{ and } a_n\geq0\}.$$ The set of simple roots associated to this Weyl chamber is $$\Pi=\{\eps_i-\eps_{i+1}:i=1,\ldots,n-1\}\cup\{\eps_n\}.$$ The barycenter of the Weyl chamber is hence $$\sf{bar}_{\frak a^+}=\{(nt,(n-1)t,\ldots,2t,t):t\geq0\}.$$ The highest weight of the representation $\SO(n,n+1)\subset\SL(d,\R)$ is $\chi(\upla a1n)=a_1.$ Finally, for any $\t\in\Pi$ one has $$\frac{\t(\sf{bar}_{\frak a^+})}{\chi(\sf{bar}_{\frak a^+})}=\frac{t}{nt}=\frac1n=\frac2{d-1}.$$
 
\subsection*{The $\Ge$ case}

The remaining case is $d=7$ and the Zariski closure of $\rho(\pi_1\E)$ being the exceptional simple Lie group $\Ge.$ We refer the reader to Knapp's book \cite[page 692]{g2} for the following computations. In this case we have $$\frak a=\{(a_1,a_2,a_3)\in\R^3:a_1+a_2+a_3=0\},$$ a Weyl chamber is $$\frak a^+=\{(a_1,a_2,a_3):a_1\geq a_2\textrm{ and }-2a_1+a_2+a_3\geq0\}.$$ 
The set of simple roots is $$\Pi=\{\eps_1-\eps_2,-2\eps_1+\eps_2+\eps_3\},$$ and the barycenter of the Weyl chamber is hence $${\sf{bar}}_{\frak a^+}=\{(-t,-4t,5t):t\geq0\}.$$ The highest weight associated to the representation $\Ge\to\SL(7,\R)$ is $$\chi=\om_1=2(\eps_1-\eps_2)-2 \eps_1+\eps_2+ \eps_3=\eps_3-\eps_2.$$ Finally, for any $\t\in\Pi$ one has $$\frac{\t(\sf{bar}_{\frak a^+})}{\chi(\sf{bar}_{\frak a^+})}=\frac{3t}{5t+4t}=\frac13=\frac2{d-1}.$$ 

\noindent
This finishes the proof. 
\end{proof}

\bibliography{/Users/andres/Dropbox/GordoSambita/CasoHitchin/stage1}
\bibliographystyle{plain}

\author{$\ $ \\
Andr\'es Sambarino\\
  Departement de Math\'ematiques\\
Universit\'e Paris Sud,\\
  F-91405 Orsay France,\\
  \texttt{andres.sambarino@gmail.com}}

\end{document}